\documentclass[reqno,12pt]{amsart}
\usepackage{amssymb}
\usepackage{mathrsfs}
\usepackage{txfonts}
\usepackage{amsfonts}
\usepackage{amstext}
\usepackage{amssymb}
\usepackage{amsmath}
\usepackage{graphicx}
\usepackage{hyperref}
\usepackage{url}
\usepackage{tikz}
\usepackage[final]{showkeys}
\usepackage{float}

\usetikzlibrary{patterns}
\hypersetup{
        colorlinks   = true,
        citecolor    = blue,
        linkcolor    = blue,
        urlcolor     = blue
}
\newtheorem{theorem}{Theorem}[section]

\newtheorem{lemma}[theorem]{Lemma}

\theoremstyle{definition}

\newenvironment{acknowledgement}{\smallskip{\sc Acknowledgement.}\rm}{\smallskip}
\numberwithin{equation}{section}

\def \color#1{}

\def\tbigcup{\mathop{\textstyle \bigcup }}

\def\Xint#1{\mathchoice
{\XXint\displaystyle\textstyle{#1}}%
{\XXint\textstyle\scriptstyle{#1}}%
{\XXint\scriptstyle\scriptscriptstyle{#1}}%
{\XXint\scriptscriptstyle\scriptscriptstyle{#1}}%
\!\int}
\def\XXint#1#2#3{{\setbox0=\hbox{$#1{#2#3}{\int}$ }
\vcenter{\hbox{$#2#3$ }}\kern-.6\wd0}}

\def\oint{\Xint-}

\def\FRAME#1#2#3#4#5#6#7#8
{
 \begin{figure}[ptbh]
 \begin{center}
 \includegraphics[height=#3]{#7}
 \caption{#5}
 \label{#6}
 \end{center}
 \end{figure}
}

\def\enddoc{\end{document}}

\allowdisplaybreaks
\input{tcilatex}

\begin{document}
\title{A positive answer to a symmetry conjecture on homogeneous IFS}
\author[Zhang]{Junda Zhang}
\address{School of Mathematics, South China University of Technology, Guangzhou 510641, China.}
\email{summerfish@scut.edu.cn}
\date{\today }

\begin{abstract}
Let $\Phi$ and $\Psi$ be two homogeneous IFSs satisfying the OSC, with opposite common contraction factors, and having the same attractor $K\subset\mathbb{R}$. We show that $K$ is symmetric. This answers a question of Feng and Wang [Adv. Math. 222 (2009), 1964–1981].
\end{abstract}

\subjclass[2020]{28A80}
\keywords{generating IFS, homogeneous, open set
condition}
\maketitle

\section{Introduction}
In this paper, a finite set $\Phi =\{\phi _{i}\}_{i=1}^{m}(m\geq 2)$ of
distinct contractive similarities in $\mathbb{R}^{n}$ is
called a (standard) \emph{iterated function system} (IFS). The \emph{attractor} of $%
\Phi $ is the unique nonempty compact subset $K\subset \mathbb{R}^{n}$ such
that
\begin{equation}
K=\Phi (K):=\tbigcup_{i=1}^{m}\phi _{i}(K).  \label{at}
\end{equation} We call $\Phi$ a \emph{generating IFS of }$K$.
In this paper, we always assume that $K\subset \mathbb{R}$ is not a singleton. We call $\Phi $
\emph{homogeneous} if each map is in the form $\phi _{i}(x)=rx+b_{i}$ for a
common contraction factor $r\in (-1,1)\setminus \{0\}$ and some different translations $b_{i}\in \mathbb{R}$, and in this case we write  $\Phi=rx+\{b_{i}\}_{i=1}^{m}. $

Separation conditions for IFSs are often required. We say that an IFS $\Phi
=\{\phi _{i}\}_{i=1}^{m}$ satisfies the \emph{open set condition} (\textrm{%
OSC}), if there exists a nonempty bounded open set $U\subseteq \mathbb{R}%
^{n} $, such that
\begin{equation}
\tbigcup_{i=1}^{m}\phi _{i}(U)\subseteq U  \label{OSC}
\end{equation}%
with the union disjoint. If further $U$ can be taken to be a convex set in (\ref{OSC}), we say that the \emph{convex open set condition} (\textrm{COSC}) is satisfied. When
the union in (\ref{at}) is disjoint, we say that $\Phi $ satisfies the \emph{%
strong separation condition} (SSC).

A fundamental problem in fractal geometry is that, given a fixed $K\subset \mathbb{R}^{n}$, what can be said about the structure of its generating IFSs, with or without separation conditions? Feng and
Wang \cite{FengWang2009} initiated this study for $K\subset \mathbb{R}$ under the OSC. They proved that
all homogeneous generating IFSs with the OSC form a finitely-generated semigroups (if non-empty) when equipped with
the composition, and also gave some sufficient conditions for these semigroups to
have a minimal element (with or without homogeneity). They also provided some special examples where a
minimal generating IFS does not exist. Later, Deng and Lau generalised the finitely-generated property for $K\subset \mathbb{R}^{n}$ under homogeneity and the SSC in \cite{DengLau2013}, and then relaxed the SSC to the OSC in \cite{DengLau2017}. Some further results on specific classes of self-similar sets were given, for
example, on two connected fractals \cite{LiYao2016} and on a construction
with complete overlaps \cite{KongYao2021}.

We are concerned with the following open problem \cite[Open Question 1]{FengWang2009}: ``Let $\Phi$ and $\Psi$ be two homogeneous IFSs satisfying the OSC, with opposite common contraction factors, and having the same attractor $K\subset\mathbb{R}$. Is it true that $K$ is symmetric?"

This question was partially answered in \cite[Lemma 3.4]{FengWang2009} under the COSC, and in \cite[Theorem 1.5]{XiaoETDS} under the SSC. But under a weaker condition OSC, different approach is required. If the answer is positive, then the statements in the finitely-generated theorem \cite[Theorem 1.3]{FengWang2009} will be much cleaner, as stated in \cite[Section 6]{FengWang2009}.

Our result positively answers this question.
\begin{theorem}
Let $\Phi$ and $\Psi$ be two homogeneous IFSs satisfying the OSC, with opposite common contraction factors, and having the same attractor $K\subset\mathbb{R}$. Then $K$ is symmetric.\label{1.1}\end{theorem}

Although our proof seems quite short, we have tried several different strategies with tools from many areas to tackle this problem. It turns out that other strategies are much more involved, lengthy but unsatisfying.

The structure of this paper is simple. We first state two lemmas used in our proof, and then present the proof of Theorem  \ref{1.1}.
We will frequently use the following notation.
Let $A = \{a_i\}_{i=1}^n$ and $B = \{b_j\}_{j=1}^m$ be two multisets of real numbers, and $r$ be a real number, then
$$rA:=\{ra_i: 1 \le i\le n\},\ A+B:=\{a_i + b_j : 1 \le i\le n ,1 \le j\le m\}.$$

\section{Two Lemmas}
We first state an easy but useful one.
\begin{lemma}\label{2.1}
Let $A = \{a_i\}_{i=1}^n$ and $B = \{b_i\}_{i=1}^n$ be two multisets of real numbers satisfying that
$$A + rA = B - rB,$$
where $r > 0$. Suppose further that
$b_i - a_i$ equals to the same number $C$ for all $1\leq i\leq n.$
Then
$$A = \dfrac{C(1-r)}{r} - A.$$
\end{lemma}
\begin{proof}
Define the generating functions
\begin{equation}\label{def}
  A(x) = \sum_{i=1}^n x^{a_i}, \qquad B(x) = \sum_{i=1}^n x^{b_i}.
\end{equation}
Since $b_i = a_i + C$ for all $1\leq i\leq n$, we have  $$B(x) = x^C A(x),$$
and so\begin{equation}\label{cc}
B(x) B(x^{-r}) = \bigl( x^C A(x) \bigr) \bigl( x^{-rC} A(x^{-r}) \bigr) = x^{C(1-r)} A(x) A(x^{-r}).
\end{equation}
The multiset equality $A + rA = B - rB$ implies that
$$A(x) A(x^r) = B(x) B(x^{-r}).$$
Plugging (\ref{cc}) into the above equation gives that
$$A(x) A(x^r) = x^{C(1-r)} A(x) A(x^{-r}).$$
Since $A(x)$ is nonzero, it follows that
$$A(x^r) = x^{C(1-r)} A(x^{-r}).$$
Letting $t = x^r$, so that $x = t^{1/r}$ and
$$A(t) = t^{C(1-r)/r} A(t^{-1}).$$
It follows by definition (\ref{def}) that
$$\sum_{i=1}^n t^{a_i} = t^{C(1-r)/r} \sum_{i=1}^n t^{-a_i} = \sum_{i=1}^n t^{C(1-r)/r - a_i}.$$
This equality of Laurent polynomials means that the multisets of exponents on both sides coincide, that is
$$\{ a_i : i = 1,\dots,n \} = \left\{ \frac{C(1-r)}{r} - a_i : i = 1,\dots,n \right\},$$ showing the desired.
\end{proof}
The second lemma is key to the proof. Similar inclusions to (\ref{lem}) also appear in the proof of \cite[Lemma 3.4]{FengWang2009} and \cite[Theorem 1.5]{XiaoETDS}.
\begin{lemma}\label{2.2}
Let $A = \{a_i\}_{i=1}^n$ and $B = \{b_i\}_{i=1}^n$ be two sets of real numbers with elements arranged in increasing order: $a_1 < a_2 < \cdots < a_n$ and $b_1 < b_2 < \cdots < b_n$. Let $r > 0$. Assume that the set $A + rA = B - rB$ has cardinality $n^2$, that is, all elements are distinct. Further assume that both sets $rB + A$ and $-rA + B$ have cardinality $n^2$. Then for every $1 \leq i \leq n$,
\begin{equation}\label{lem}
  b_i - r b_n = a_i + r a_1.
\end{equation}
\end{lemma}

\begin{proof}
Since $A + rA = B - rB$, we may write \begin{equation}\label{ba}
                                        b_k-rb_n=a_{s(k)}+ra_{t(k)},\ a_k+ra_1=b_{u(k)}-rb_{v(k)}.
                                      \end{equation}
We claim that $s$ and $u$ are permutations from 1 to $n$.

Rewrite $b_k - r b_n = a_{s(k)} + r a_{t(k)}$ in (\ref{ba}) as
$$b_k - r a_{t(k)} = a_{s(k)} + r b_n.$$
Suppose that there exist $k \neq k'$ such that $s(k) = s(k')$. Then
$$a_{s(k)} + r b_n = a_{s(k')} + r b_n,$$
so we have
$$b_k - r a_{t(k)} = b_{k'} - r a_{t(k')}.$$
Since the set $-rA + B$ has $n^2$ distinct elements (the mapping $(i,j) \mapsto -r a_i + b_j$ is injective), it follows that $t(k) = t(k')$ and $k = k'$, a contradiction.
Thus $s$ is injective, therefore a permutation.
The fact that $u$ is a bijection follows similarly.
Rewrite $a_k + r a_1 = b_{u(k)} - r b_{v(k)}$ in (\ref{ba})  as
$$b_{u(k)} - r a_1 = a_k + r b_{v(k)}.$$
Suppose that there exist $k \neq k'$ such that $u(k) = u(k')$, then
$$a_k + r b_{v(k)} =b_{u(k)} - r a_1 = b_{u(k')} - r a_1= a_{k'} + r b_{v(k')}.$$
Since the set $A + rB$ has $n^2$ distinct elements,
it follows that $(k, v(k)) = (k', v(k'))$, a contradiction.

Now we prove (\ref{lem}) by induction on $i$.

\noindent \textbf{Base case: } $i = 1$.
Since the smallest element of $A + rA = \{a_i + r a_j : 1 \leq i,j \leq n\}$ is $a_1 + r a_1$ and the smallest element of $B - rB = \{b_i - r b_j : 1 \leq i,j \leq n\}$ is $b_1 - r b_n$, it follows from the condition $A + rA = B - rB$ that
\[
b_1 - r b_n = a_1 + r a_1,
\] showing the base case.

\noindent \textbf{Inductive step:}
Assume that for all $j < i$, we have
$$b_j - r b_n = a_j + r a_1.$$
We prove the statements for step $i$.
Indeed, recall $a_i + r a_1 = b_{u(i)} - r b_{v(i)}$ in (\ref{ba}). Since $u$ is a bijection by the above claim and $u(j)=j$ for all $j < i$ by the inductive hypothesis, we have $u(i) \ge i$. Hence, $b_{u(i)} \ge b_i$ by the increasing order assumption, and so
$$a_i + r a_1 = b_{u(i)} - r b_{v(i)} \ge b_i - r b_{v(i)}\ge b_i - r b_{n}.$$
Similarly, recall $b_i - r b_n = a_{s(i)} + r a_{t(i)}$ in (\ref{ba}). Since $s$ is a bijection and $s(j)=j$ for $j < i$, we have $s(i) \ge i$. Hence, $a_{s(i)} \ge a_i$, and so
$$b_i - r b_n = a_{s(i)} + r a_{t(i)} \ge a_i + r a_{t(i)}\ge a_i + r a_{1}.$$
This completes the inductive step.
By induction, (\ref{lem}) holds for all $i = 1, \dots, n$.
\end{proof}
\section{Proof of the Theorem}
We are now in a position to prove our theorem.
\begin{proof}[Proof of Theorem \ref{1.1}] We write $\Phi=rx+A$ and $\Psi=-rx+B$, where $r>0$. By the Moran equation, we can write $A = \{a_i\}_{i=1}^n$ and $B = \{b_i\}_{i=1}^n$ with elements arranged in increasing order, and we know that $$-r^2x+rB + A=\Phi\circ \Psi=\Psi \circ \Phi=-r^2x-rA + B$$ satisfies the OSC by \cite[Proposition 2.2]{FengWang2009}, thus the set $$rB + A=-rA + B$$ has cardinality $n^2$ (with all elements distinct). By \cite[Proposition 2.1]{FengWang2009},
we know that $$r^2x+A + rA=\Phi\circ \Phi=\Psi \circ \Psi=r^2x+B - rB$$ satisfies the OSC, thus the set $$A + rA = B - rB$$has cardinality $n^2$. It follows by Lemma \ref{2.2} that, for every $1 \leq i \leq n$,
\[
b_i - r b_n = a_i + r a_1.
\]
Then by using Lemma \ref{2.1}, taking $C=r b_n +r a_1,$ we know that $$A=\frac{C(1-r)}{r}-A,$$
showing that $K$ is symmetric (see for example, the last line in \cite[Proof of
Lemma 3.4]{FengWang2009}, which states that the attractor of a homogeneous
IFS $\rho x+A$ is symmetric when $A$ is symmetric).
\end{proof}
\bibliographystyle{siam}
\bibliography{ref1}
\begin{acknowledgement}
The author thanks Liming Ling and Jian-Ci Xiao for valuable suggestions.
\end{acknowledgement}

\end{document}